\documentclass[11pt]{article}
       \usepackage{amsfonts}
       \usepackage{stmaryrd}
       \usepackage{latexsym,amssymb,mathrsfs,fancyhdr}
       \font\tenmsb=msbm10
       \font\sevenmsb=msbm7
       \font\fivemsb=msbm5
       \catcode`\@=11
       \ifx\amstexloaded@\relax\catcode`\@=\active
       \endinput\else\let\amstexloaded@\relax\fi
       \def\spaces@{\space\space\space\space\space}
       \def\spaces@@{\spaces@\spaces@\spaces@\spaces@\spaces@}
       \def\space@.  {\futurelet\space@\relax}
       \space@.   %
       \def\Err@#1{\errhelp\defaulthelp@\errmessage{AmS-teX error: #1}}
       \def\relaxnext@{\let\next\relax}
       \def\accentfam@{7}
       \def\noaccents@{\def\accentfam@{0}}
       \def\Cal{\relaxnext@\ifmmode\let\next\Cal@\else
       \def\next{\Err@{Use \string\Cal\space only in math mode}}\fi\next}
       \def\Cal@#1{{\Cal@@{#1}}}
       \def\Cal@@#1{\noaccents@\fam\tw@#1}
       \def\Bbb{\relaxnext@\ifmmode\let\next\Bbb@\else
       \def\next{\Err@{Use \string\Bbb\space only in math mode}}\fi\next}
       \def\Bbb@#1{{\Bbb@@{#1}}}
       \def\Bbb@@#1{\noaccents@\fam\msbfam#1}
       \def\co{\tiny{\textcircled{\tiny\#}}}
       \newfam\msbfam
       \textfont\msbfam=\tenmsb
       \scriptfont\msbfam=\sevenmsb
       \scriptscriptfont\msbfam=\fivemsb
\usepackage{dsfont}
\usepackage[german,english]{babel}
\usepackage{amsmath,amssymb}
\usepackage[square, comma, sort&compress, numbers]{natbib}
\usepackage{cases}
\usepackage{mathrsfs}
\usepackage{amsmath}
\usepackage{amsthm}
\usepackage{amsfonts}
\usepackage{amssymb}
\usepackage{latexsym}
\usepackage{fancyhdr}
\usepackage{extarrows}

\newtheorem{thm}{Theorem}[section]

\newtheorem{lem}[thm]{Lemma}
\newtheorem{rem}[thm]{Remark}
\newtheorem{iteration lemma}[thm]{iteration Lemma}
\newtheorem{cor}[thm]{Corollary}

\newtheorem*{acknowledgements*}{ACKNOWLEDGEMENtS}

\begin{document}

\setlength{\columnsep}{5pt}
\title{\bf The core and dual core inverses of morphisms with kernels}
\author{Tingting  Li\footnote{ E-mail: littnanjing@163.com},
\ Jianlong Chen\footnote{ Corresponding author. E-mail: jlchen@seu.edu.cn},
\ Mengmeng Zhou\footnote{ E-mail: mmz9209@163.com}\\
School of  Mathematics, Southeast University \\ Nanjing, Jiangsu 210096,  China\\\\
\ Dingguo Wang\footnote{ E-mail: dingguo95@126.com}\\
School of  Mathematical Sciences, Qufu Normal University \\
 Qufu, Shandong 273165,  China}
     \date{}

\maketitle
\begin{quote}
{\textbf{}\small
Let $\mathscr{C}$ be an additive category with an involution $\ast$.
Suppose that
$\varphi : X \rightarrow X$ is a morphism with kernel $\kappa : K \rightarrow X$ in $\mathscr{C}$,
then $\varphi$ is core invertible if and only if $\varphi$ has a cokernel $\lambda: X \rightarrow L$ and both $\kappa\lambda$ and $\varphi^{\ast}\varphi^3+\kappa^{\ast}\kappa$ are invertible.
In this case,
we give the representation of the core inverse of $\varphi$.
We also give the corresponding result about dual core inverse.


}
\end{quote}

\section{ Introduction }\label{a}
Throughout this paper,
$\mathscr{C}$ is an additive category with an involution $\ast$,
that is to say,
there is a unary operation $\ast$ on the morphisms such that
$\varphi : X \rightarrow Y$ implies $\varphi^{\ast} : Y \rightarrow X$ and that $(\varphi^{\ast})^{\ast}=\varphi,
(\varphi\psi)^{\ast}=\psi^{\ast}\varphi^{\ast}$ for any $\psi : Y \rightarrow Z$ and $(\varphi+\phi)^{\ast}=\varphi^{\ast}+\phi^{\ast}$ for any $\phi : X \rightarrow Y$.
(See for example, \cite[p. 131]{PR2}.)
And $R$ is a $\ast$-ring,
which is an associative ring with 1 and an involution $\ast$.

Let $\varphi : X \rightarrow Y$ be a morphism of $\mathscr{C}$,
we say that $\varphi$ is regular (or $\{1\}$-invertible) if there is a morphism $\chi : Y \rightarrow X$ in $\mathscr{C}$ such that $\varphi\chi\varphi=\varphi$.
In this case,
$\chi$ is said to be an inner inverse of $\varphi$ and is denoted by $\varphi^{-}$.
If such a regular element $\varphi$ also satisfies $\chi\varphi\chi=\chi$,
then we call that $\chi$ is a reflexive inverse of $\varphi$.
When $X=Y$,
if $\chi$ is a reflexive inverse of $\varphi$ and commutes with $\varphi$,
then $\varphi$ is group invertible and such a $\chi$ is called the group inverse of $\varphi$.
The group inverse of $\varphi$ is unique if it exists and is denoted by $\varphi^{\#}$.

Recall that $\varphi$ is Moore-Penrose invertible if there is a morphism $\chi : Y \rightarrow X$ in $\mathscr{C}$ satisfying the following four equations:
\begin{center}
  $(1)$ $\varphi\chi\varphi=\varphi$,~~~$(2)$ $\chi\varphi\chi=\chi$,~~~$(3)$ $(\varphi\chi)^{\ast}=\varphi\chi$,~~~$(4)$ $(\chi\varphi)^{\ast}=\chi\varphi$.
\end{center}
If such a $\chi$ exists,
then it is unique and denoted by $\varphi^{\dagger}$.
Let $\varphi\{i,j,\cdots,l\}$ denote the set of morphisms $\chi$ which satisfy equations $(i),(j),\cdots,(l)$ from among equations $(1)$-$(4)$,
and in this case,
$\chi$ is called the $\{i,j,\cdots,l\}$-inverse of $\varphi$.
If $\chi \in \varphi\{1, 3\}$,
then $\chi$ is called a $\{1, 3\}$-inverse of $\varphi$ and is denoted by $\varphi^{(1, 3)}$.
A $\{1, 4\}$-inverse of $\varphi$ can be similarly defined.
Also,
a regular element and a reflexive invertible element can be called a $\{1\}$-invertible element and a $\{1, 2\}$-invertible element,
respectively.

Baksalary and Trenkler \cite{OM} introduced the core and dual core inverses for a complex matrix.
Then,
Raki\'{c} et al. \cite{DSR} generalized this concept to an arbitrary $\ast$-ring,
and they use five equations to characterize the core inverse.
Later,
Xu et al. \cite{XSZ} proved that these five equations can be dropped to three equations.
In the following,
we rewrite these three equations in the category case.
Let $\varphi : X \rightarrow X$ be a morphism of $\mathscr{C}$,
if there is a morphism $\chi : X \rightarrow X$ satisfying
\begin{equation*}
\begin{split}
  ~(\varphi\chi)^{\ast}=\varphi\chi,~\varphi\chi^2=\chi,~\chi\varphi^2=\varphi,
\end{split}
\end{equation*}
then $\varphi$ is core invertible and $\chi$ is called the core inverse of $\varphi$.
If such $\chi$ exists,
then it is unique and denoted by $\varphi^{\co}$.
And the dual core inverse can be given dually and denoted by $\varphi_{\co}$.

Group inverses and Moore-Penrose inverses of morphisms were investigated some years ago.
(See,\cite{PR2} and \cite{DW}-\cite{PR3}.)
In \cite{DW},
Robinson and Puystjens give the characterizations about the Moore-Penrose inverse and the group inverse of a morphism with kernels.
In \cite{MR},
Miao and Robinson investigate the group and Moore-Penrose inverses of regular morphisms with kernel and cokernel.
Inspired by them,
we consider the core invertibility and dual core invertibility of a morphism with kernels and give their representations.
In the process of proving the above results,
we obtain some characterizations for core inverse and dual core inverse of an element in a $\ast$-ring by the properties of annihilators and units.

The following notations will be used in this paper:
$aR=\{ax~|~x\in R\}$,
$Ra=\{xa~|~x\in R\}$,
$^{\circ}\!a=\{x\in R~|~xa=0\}$,
$a^{\circ}=\{x\in R~|~ax=0\}$,
$R^{\co}=\{a\in R~|~a$ is core invertible$\}$,
$R_{\co}=\{a\in R~|~a$ is dual core invertible$\}$.
Before beginning,
there are some lemmas presenting for the further reference.
It should be pointed out,
the following Lemma~\ref{222} - \ref{core-inverse 2} were put forward in a $\ast$- ring.
It is easy to prove that they are valid in an additive category with an involution $\ast$.
Thus,
we rewrite them in the category case.

\begin{lem} \cite[Theorem $2.6$ and $2.8$]{XSZ}\label{222}
Let $\varphi : X \rightarrow X$ be a morphism of $\mathscr{C}$,
we have the following results:\\
(1) $\varphi$ is core invertible if and only if $\varphi$ is group invertible and $\{1,3\}$-invertible.
In this case,
$\varphi^{\co}=\varphi^{\#}\varphi\varphi^{(1,3)}$.\\
(2) $\varphi$ is dual core invertible if and only if $\varphi$ is group invertible and $\{1,4\}$-invertible.
In this case,
$\varphi_{\co}=\varphi^{(1,4)}\varphi\varphi^{\#}$.
\end{lem}

\begin{lem} \cite[p. $201$]{Hartwig}\label{000}
Let $\varphi : X \rightarrow Y$ be a morphism of $\mathscr{C}$,
we have the following results:\\
$(1)$ $\varphi$ is $\{1,3\}$-invertible with $\{1,3\}$-inverse $\chi : Y \rightarrow X$ if and only if $\chi^{\ast}\varphi^{\ast}\varphi=\varphi;$\\
$(2)$ $\varphi$ is $\{1,4\}$-invertible with $\{1,4\}$-inverse $\zeta : Y \rightarrow X$ if and only if $\varphi\varphi^{\ast}\zeta^{\ast}=\varphi.$
\end{lem}

\begin{lem}\cite[Theorem $2.10$]{Li}\label{core-inverse 2}
Let $\varphi : X \rightarrow X$ be a morphism of $\mathscr{C}$ and $n\geqslant 2$ a positive integer,
we have the following results:\\
(i) $\varphi$ is core invertible if and only if there exist morphisms $\varepsilon : X \rightarrow X$ and $\tau : X \rightarrow X$ such that $\varphi=\varepsilon(\varphi^{\ast})^n\varphi=\tau\varphi^n$.
In this case,
$\varphi^{\co}=\varphi^{n-1}\varepsilon^{\ast}$.\\
(ii) $\varphi$ is dual core invertible if and only if there exist morphisms $\theta : X \rightarrow X$ and $\rho : X \rightarrow X$ such that $\varphi=\varphi(\varphi^{\ast})^n\theta=\varphi^n\rho$.
In this case,
$\varphi_{\co}=\theta^{\ast}\varphi^{n-1}$.
\end{lem}

\begin{lem} \cite[Proposition $7$]{Hartwig}\label{group-inverse}
Let $a\in R$.
$a\in R^{\#}$ if and only if $a=a^{2}x=ya^{2}$ for some $x, y\in R$.
In this case, $a^{\#}=yax=y^{2}a=ax^{2}$.
\end{lem}

\section{The Core and Dual Core Inverse of a Morphism with Kernel}\label{b}
In \cite{DW},
Robinson and Puystjens gave the characterizations about the Moore-Penrose inverse and the group inverse of a morphism with kernels,
see the following two lemmas.

\begin{lem}\cite[Theorem $1$]{DW}
Let $\varphi : X \rightarrow Y$ be a morphism in $\mathscr{C}$.
If $\kappa : K \rightarrow X$ is a kernel of $\varphi$,
then $\varphi$ has a Moore-Penrose inverse $\varphi^{\dagger}$ with respect to $\ast$ if and only if
$$\varphi\varphi^{\ast}+\kappa^{\ast}\kappa : X \rightarrow X$$
is invertible.
In this case,
$\kappa$ also has a Moore-Penrose inverse $\kappa^{\dagger}$,
$\kappa\kappa^{\ast}$ is invertible,
$$\kappa^{\dagger}=\kappa^{\ast}(\kappa\kappa^{\ast})^{-1}=(\varphi\varphi^{\ast}+\kappa^{\ast}\kappa)^{-1}\kappa^{\ast},$$
and
$$\varphi^{\dagger}=\varphi^{\ast}(\varphi\varphi^{\ast}+\kappa^{\ast}\kappa)^{-1}.$$
Dually,
if $\lambda : Y \rightarrow L$ is a cokernel of $\varphi$,
then $\varphi$ has a Moore-Penrose inverse $\varphi^{\dagger}$ with respect to $\ast$ if and only if
$$\varphi^{\ast}\varphi+\lambda\lambda^{\ast} : Y \rightarrow Y$$
is invertible.
In this case,
$\lambda$ also has a Moore-Penrose inverse $\lambda^{\dagger}$,
$\lambda^{\ast}\lambda$ is invertible,
$$\lambda^{\dagger}=(\lambda^{\ast}\lambda)^{-1}\lambda^{\ast}=\lambda^{\ast}(\varphi^{\ast}\varphi+\lambda\lambda^{\ast})^{-1},$$
and
$$\varphi^{\dagger}=(\varphi^{\ast}\varphi+\lambda\lambda^{\ast})^{-1}\varphi^{\ast}.$$
\end{lem}

\begin{lem}\cite[Corollary $2$]{DW}\label{kernel-group inverse}
Let $\varphi : X \rightarrow X$ be a morphism in $\mathscr{C}$.
If $\kappa : K \rightarrow X$ is a kernel of $\varphi$,
then $\varphi$ has a group inverse if and only if $\varphi$ has a cokernel $\lambda : X \rightarrow L$ and both $\kappa\lambda : K \rightarrow L$ and $\varphi^2+\lambda(\kappa\lambda)^{-1}\kappa : X \rightarrow X$ are invertible.
In this case,
$\gamma=\lambda(\kappa\lambda)^{-1} : X \rightarrow K$ is a cokernel of $\varphi$,
$\varphi\varphi^{\#}+\gamma\kappa=1_X$,
and
$$\varphi^{\#}=\varphi(\varphi^2+\gamma\kappa)^{-1}=(\varphi^2+\gamma\kappa)^{-1}\varphi.$$
\end{lem}

There are some papers characterizing the core and dual core inverse by units.
(See for example,
\cite{Chen} and \cite{Li}.)
Inspired by them and the above two lemmas,
we get characterizations of the core invertibility of a morphism with kernel.

\begin{thm}\label{kernel}
Let $\varphi : X \rightarrow X$ be a morphism in $\mathscr{C}$.
If $\kappa : K \rightarrow X$ is a kernel of $\varphi$,
then $\varphi$ has a core inverse in $\mathscr{C}$ if and only if $\varphi$ has a cokernel $\lambda : X \rightarrow L$ and both $\kappa\lambda : K \rightarrow L$ and $\varphi^{\ast}\varphi^3+\kappa^{\ast}\kappa : X \rightarrow X$ are invertible.
In this case,
$\gamma=\lambda(\kappa\lambda)^{-1} : X \rightarrow K$ is a cokernel of $\varphi$,
$\varphi^{\co}\varphi+\gamma\kappa=1_X$,
and
\begin{equation*}
\begin{split}
     \varphi^{\co}=\varphi^2(\varphi^{\ast}\varphi^3+\kappa^{\ast}\kappa)^{-1}\varphi^{\ast}.
\end{split}
\end{equation*}
\end{thm}

\begin{proof}
Let $\lambda : X \rightarrow L$ be a cokernel of $\varphi$ with both $\kappa\lambda$ and $\varphi^{\ast}\varphi^3+\kappa^{\ast}\kappa$ invertible,
and set $\gamma=\lambda(\kappa\lambda)^{-1}$.
Since  $\kappa\varphi=0=\varphi\lambda$ and $\varphi^{\ast}\kappa^{\ast}=0=\lambda^{\ast}\varphi^{\ast}$,
then
\begin{equation*}
\begin{split}
     \varphi^{\ast}\varphi^3+\kappa^{\ast}\kappa=(\varphi^{\ast}\varphi+\kappa^{\ast}\kappa)(\varphi^2+\gamma\kappa),
\end{split}
\end{equation*}
because $\varphi^{\ast}\varphi+\kappa^{\ast}\kappa$ is symmetric,
both $\varphi^{\ast}\varphi+\kappa^{\ast}\kappa : X \rightarrow X$ and $\varphi^2+\gamma\kappa : X \rightarrow X$ are invertible.
In addition,
\begin{equation*}
\begin{split}
     (\varphi^{\ast}\varphi+\kappa^{\ast}\kappa)\varphi=\varphi^{\ast}\varphi^2,
\end{split}
\end{equation*}
and for $s\geq 0$ an integer,
\begin{equation*}
\begin{split}
     \varphi^{1+s}(\varphi^2+\gamma\kappa)=\varphi^{3+s}=(\varphi^2+\gamma\kappa)\varphi^{1+s}.
\end{split}
\end{equation*}
Consequently,
\begin{equation}\label{D5}
\begin{split}
     \varphi=(\varphi^{\ast}\varphi+\kappa^{\ast}\kappa)^{-1}\varphi^{\ast}\varphi^2,
\end{split}
\end{equation}
\begin{equation}\label{D6}
\begin{split}
     \varphi^{1+s}=\varphi^{3+s}(\varphi^2+\gamma\kappa)^{-1}=(\varphi^2+\gamma\kappa)^{-1}\varphi^{3+s},
\end{split}
\end{equation}
\begin{equation}\label{D7}
\begin{split}
     \varphi^{1+s}(\varphi^2+\gamma\kappa)^{-1}=(\varphi^2+\gamma\kappa)^{-1}\varphi^{1+s}.
\end{split}
\end{equation}
Let $\chi=\varphi^2(\varphi^{\ast}\varphi^3+\kappa^{\ast}\kappa)^{-1}\varphi^{\ast}=\varphi^2(\varphi^2+\gamma\kappa)^{-1}(\varphi^{\ast}\varphi+\kappa^{\ast}\kappa)^{-1}\varphi^{\ast}$,
we now show that $\chi$ is the core inverse of $\varphi$.
Since
\begin{equation*}
\begin{split}
     \varphi\chi
     &~= \varphi\varphi^2(\varphi^2+\gamma\kappa)^{-1}(\varphi^{\ast}\varphi+\kappa^{\ast}\kappa)^{-1}\varphi^{\ast}\\
     &~= [\varphi^3(\varphi^2+\gamma\kappa)^{-1}](\varphi^{\ast}\varphi+\kappa^{\ast}\kappa)^{-1}\varphi^{\ast}\\
     &~\stackrel{(\ref{D6})}{=} \varphi(\varphi^{\ast}\varphi+\kappa^{\ast}\kappa)^{-1}\varphi^{\ast},
\end{split}
\end{equation*}
thus $(\varphi\chi)^{\ast}=\varphi\chi$.
In addition,
\begin{equation*}
\begin{split}
     \chi\varphi^2
     &~= \varphi^2(\varphi^2+\gamma\kappa)^{-1}(\varphi^{\ast}\varphi+\kappa^{\ast}\kappa)^{-1}\varphi^{\ast}\varphi^2 ~\stackrel{(\ref{D5})}{=} \varphi^2(\varphi^2+\gamma\kappa)^{-1}\varphi\\
     &~\stackrel{(\ref{D7})}{=} \varphi^2\varphi(\varphi^2+\gamma\kappa)^{-1} ~\stackrel{(\ref{D6})}{=} \varphi,
\end{split}
\end{equation*}
and
\begin{equation*}
\begin{split}
     \varphi\chi^2
     &~= \varphi\varphi^2(\varphi^2+\gamma\kappa)^{-1}(\varphi^{\ast}\varphi+\kappa^{\ast}\kappa)^{-1}\varphi^{\ast}\varphi^2(\varphi^2+\gamma\kappa)^{-1}(\varphi^{\ast}\varphi+\kappa^{\ast}\kappa)^{-1}\varphi^{\ast}\\
     &~= [\varphi^3(\varphi^2+\gamma\kappa)^{-1}][(\varphi^{\ast}\varphi+\kappa^{\ast}\kappa)^{-1}\varphi^{\ast}\varphi^2](\varphi^2+\gamma\kappa)^{-1}(\varphi^{\ast}\varphi+\kappa^{\ast}\kappa)^{-1}\varphi^{\ast}\\
     &~\stackrel{(\ref{D5})(\ref{D6})}{=} \varphi\varphi(\varphi^2+\gamma\kappa)^{-1}(\varphi^{\ast}\varphi+\kappa^{\ast}\kappa)^{-1}\varphi^{\ast} = \chi.
\end{split}
\end{equation*}
Therefore,
$\varphi$ is core invertible with core inverse $\varphi^{\co}=\varphi^2(\varphi^{\ast}\varphi^3+\kappa^{\ast}\kappa)^{-1}\varphi^{\ast}$.

Conversely,
suppose that $\varphi$ has a core inverse $\varphi^{\co}$,
then $\varphi$ is group invertible and $\varphi^{\#}\varphi=\varphi^{\co}\varphi$ by Lemma~\ref{222}.
Therefore,
by applying Lemma~\ref{kernel-group inverse},
$\varphi$ has a cokernel $\lambda : X \rightarrow L$,
both $\kappa\lambda : K \rightarrow L$ and $\varphi^2+\lambda(\kappa\lambda)^{-1}\kappa : X \rightarrow X$ are invertible and $1_X=\varphi\varphi^{\#}+\gamma\kappa=\varphi^{\#}\varphi+\gamma\kappa=\varphi^{\co}\varphi+\gamma\kappa$,
where $\gamma=\lambda(\kappa\lambda)^{-1} : X \rightarrow K$ is a cokernel of $\varphi$.
In addition,
since $\kappa\varphi=0=\varphi\lambda$,
thus $\kappa\varphi^{\co}=\kappa\varphi\varphi^{\co}\varphi^{\co}=0$,
$\varphi\gamma=0$ and $\kappa\gamma=1_K$,
furthermore,
\begin{equation*}
\begin{split}
     &~ ~~~(\varphi^{\co}(\varphi^{\co})^{\ast}+\gamma\gamma^{\ast})(\varphi^{\ast}\varphi+\kappa^{\ast}\kappa)\\
     &~= \varphi^{\co}(\varphi^{\co})^{\ast}\varphi^{\ast}\varphi+\varphi^{\co}(\varphi^{\co})^{\ast}\kappa^{\ast}\kappa+\gamma\gamma^{\ast}\varphi^{\ast}\varphi+\gamma\gamma^{\ast}\kappa^{\ast}\kappa\\
     &~= \varphi^{\co}(\varphi\varphi^{\co})^{\ast}\varphi+\varphi^{\co}(\kappa\varphi^{\co})^{\ast}\kappa+\gamma(\varphi\gamma)^{\ast}\varphi+\gamma(\kappa\gamma)^{\ast}\kappa\\
     &~= \varphi^{\co}\varphi+\gamma\kappa=1_X.
\end{split}
\end{equation*}
Since $\varphi^{\ast}\varphi+\kappa^{\ast}\kappa$ is symmetric,
it follows that $\varphi^{\ast}\varphi+\kappa^{\ast}\kappa$ is invertible with inverse $\varphi^{\co}(\varphi^{\co})^{\ast}+\gamma\gamma^{\ast}$.
Consequently,
$\varphi^{\ast}\varphi^3+\kappa^{\ast}\kappa=(\varphi^{\ast}\varphi+\kappa^{\ast}\kappa)(\varphi^2+\lambda(\kappa\lambda)^{-1}\kappa)$ is invertible.
\end{proof}

Dually,
we obtain the following result.

\begin{thm}\label{cokernel}
Let $\varphi : X \rightarrow X$ be a morphism of an additive category $\mathscr{C}$.
If $\lambda : X \rightarrow L$ is a cokernel of $\varphi$,
then $\varphi$ has a dual core inverse in $\mathscr{C}$ if and only if $\varphi$ has a kernel $\kappa : K \rightarrow X$ and both $\kappa\lambda : K \rightarrow L$ and $\varphi^3\varphi^{\ast}+\lambda\lambda^{\ast} : X \rightarrow X$ are invertible.
In this case,
$\delta=(\kappa\lambda)^{-1}\kappa : L \rightarrow X$ is a kernel of $\varphi$,
$\varphi_{\co}\varphi+\delta\kappa=1_X$,
and
\begin{equation*}
\begin{split}
     \varphi_{\co}=\varphi^{\ast}(\varphi^3\varphi^{\ast}+\lambda\lambda^{\ast})^{-1}\varphi^2.
\end{split}
\end{equation*}
\end{thm}

\begin{rem}
In fact,
one can easily find that Theorem~\ref{kernel} is true when we raise the $3$ power to $n$ power,
that is to say,
change $\varphi^{\ast}\varphi^3+\kappa^{\ast}\kappa$ to $\varphi^{\ast}\varphi^n+\kappa^{\ast}\kappa$,
where $n\geqslant 3$.
And in this case,
$\varphi^{\co}=\varphi^{n-1}(\varphi^{\ast}\varphi^n+\kappa^{\ast}\kappa)^{-1}\varphi^{\ast}$.
Similarly,
it is valid for dual core inverse.
\end{rem}

Consider Theorem~\ref{kernel} in the ring case,
we obtain the following result.

\begin{thm}
Let $a\in R$ and $n\geqslant 3$ a positive integer.
Then $a\in R^{\co}$ if and only if there exists $b\in R$ such that $^{\circ}\!a=Rb$ and  $u=a^{\ast}a^n+b^{\ast}b$ is invertible.
In this case,
$$a^{\co}=a^{n-1}u^{-1}a^{\ast}.$$
\end{thm}

\begin{proof}
Suppose that $a$ is core invertible with core inverse $a^{\co}$.
Let $b=1-aa^{\co}$,
then $b^{\ast}=b=b^2$ and $Rb=R(1-aa^{\co})=$ $^{\circ}\!(aa^{\co})$.
Obviously $^{\circ}\!a\subseteq$ $^{\circ}\!(aa^{\co})$;
if $xaa^{\co}=0$,
then $xa=xaa^{\co}a=0$,
hence $^{\circ}\!(aa^{\co})\subseteq$ $^{\circ}\!a$.
Therefore,
$^{\circ}\!a=$ $^{\circ}\!(aa^{\co})=Rb$.
In addition,
\begin{equation*}
\begin{split}
    u=a^{\ast}a^n+b^{\ast}b=a^{\ast}a^n+1-aa^{\co}=(a^{\ast}+1-aa^{\co})(a^n+1-aa^{\co}).
\end{split}
\end{equation*}
It is easy to verify that $(a^{\co}+1-a^{\co}a)^{\ast}$ and $(a^{\co})^n+1-a^{\co}a$ are inverses of $a^{\ast}+1-aa^{\co}$ and $a^n+1-aa^{\co}$,
respectively.
Thus $a^{\ast}+1-aa^{\co}$ and $a^n+1-aa^{\co}$ are both invertible,
which implies that $u=a^{\ast}a^n+1-aa^{\co}$ is invertible.

Conversely,
assume that there exists $b\in R$ such that $^{\circ}\!a=Rb$ and  $u=a^{\ast}a^n+b^{\ast}b$ is invertible,
where $n\geqslant 3$ is a positive integer.
Then $b\in$ $^{\circ}\!a$,
that is to say,
$ba=0=a^{\ast}b^{\ast}$.
Since $(a^{\ast})^{n-1}u=(a^{\ast})^na^n$ is symmetric,
namely,
$(a^{\ast})^{n-1}u=u^{\ast}a^{n-1}$,
which implies the following equation
\begin{equation}\label{x}
\begin{split}
    (u^{\ast})^{-1}(a^{\ast})^{n-1}=a^{n-1}u^{-1}.
\end{split}
\end{equation}
Also,
$a^{\ast}u=(a^{\ast})^2a^n$ implies $a^{\ast}=(a^{\ast})^2a^nu^{-1}$,
hence we have
\begin{equation*}
\begin{split}
    a
    &~=(u^{\ast})^{-1}(a^{\ast})^na^2=[(u^{\ast})^{-1}(a^{\ast})^{n-1}]a^{\ast}a^2\\
    &~\stackrel{(\ref{x})}{=}a^{n-1}u^{-1}a^{\ast}a^2\in a^2R\cap Ra^2,
\end{split}
\end{equation*}
so $a$ is group invertible with group inverse $a^{\#}$ according to Lemma~\ref{group-inverse}.
Since $1-a^{\#}a\in$ $^{\circ}\!a=Rb$,
thus $1-a^{\#}a=yb$ for some $y\in R$.
Pre-multiplication of $1-a^{\#}a=yb$ by $a$ and $b$ yield $ayb=0$ and $b=byb$,
respectively.
Therefore,
$u=a^{\ast}a^n+b^{\ast}b$ can be decomposed as $$u=a^{\ast}a^n+b^{\ast}b=(a^{\ast}a+b^{\ast}b)(a^{n-1}+yb).$$
Since $u$ is invertible and $a^{\ast}a+b^{\ast}b$ is symmetric,
thus both $a^{\ast}a+b^{\ast}b$ and $a^{n-1}+yb$ are invertible.
Set $x=a^{n-1}u^{-1}a^{\ast}=a^{n-1}(a^{n-1}+yb)^{-1}(a^{\ast}a+b^{\ast}b)^{-1}a^{\ast}$,
we show that $x$ is the core inverse of $a$.
Since
\begin{equation*}
\begin{split}
    ax
    &~=a^n(a^{n-1}+yb)^{-1}(a^{\ast}a+b^{\ast}b)^{-1}a^{\ast}\\
    &~=a(a^{n-1}+yb)(a^{n-1}+yb)^{-1}(a^{\ast}a+b^{\ast}b)^{-1}a^{\ast}\\
    &~=a(a^{\ast}a+b^{\ast}b)^{-1}a^{\ast}
\end{split}
\end{equation*}
shows that $(ax)^{\ast}=ax$,
\begin{equation*}
\begin{split}
    ax^2
    &~=(ax)x=a(a^{\ast}a+b^{\ast}b)^{-1}a^{\ast}a^{n-1}(a^{n-1}+yb)^{-1}(a^{\ast}a+b^{\ast}b)^{-1}a^{\ast}\\
    &~=a(a^{\ast}a+b^{\ast}b)^{-1}(a^{\ast}a+b^{\ast}b)a^{n-2}(a^{n-1}+yb)^{-1}(a^{\ast}a+b^{\ast}b)^{-1}a^{\ast}\\
    &~=aa^{n-2}(a^{n-1}+yb)^{-1}(a^{\ast}a+b^{\ast}b)^{-1}a^{\ast}=x,
\end{split}
\end{equation*}
and
\begin{equation*}
\begin{split}
    xa^2
    &~=a^{n-1}(a^{n-1}+yb)^{-1}(a^{\ast}a+b^{\ast}b)^{-1}a^{\ast}a^2\\
    &~=a^{n-1}(a^{n-1}+yb)^{-1}(a^{\ast}a+b^{\ast}b)^{-1}(a^{\ast}a+b^{\ast}b)a\\
    &~=a^{n-1}(a^{n-1}+yb)^{-1}a=(a^{n-1}+yb)^{-1}a^{n-1}a\\
    &~=(a^{n-1}+yb)^{-1}(a^{n-1}+yb)a=a,
\end{split}
\end{equation*}
thus $x=a^{\co}$.
\end{proof}

In the same way, there is a corresponding result for dual core inverse.

\begin{thm}
Let $a\in R$ and $n\geqslant 3$ a positive integer,
then $a\in R_{\co}$ if and only if there exists $c\in R$ such that $a^{\circ}=cR$ and  $v=a^na^{\ast}+cc^{\ast}$ is invertible.
In this case,
$$a_{\co}=a^{\ast}v^{-1}a^{n-1}.$$
\end{thm}

Let $\varphi : X \rightarrow Y$ be a morphism in $\mathscr{C}$.
If $\eta\varphi=0 : N \rightarrow Y$ is the zero morphism,
then we shall call the morphism $\eta : N \rightarrow X$ an annihilator of the morphism $\varphi$.
Dually,
we call $\varphi$ a coannihilator of $\eta$.
(See for example, \cite{PR3}.)

\begin{thm}
Let $\varphi : X \rightarrow X$ be a morphism in $\mathscr{C}$ and $n\geqslant 2$ a positive integer.
Then $\varphi$ is core invertible if and only if there exists an annihilator $\eta : N \rightarrow X$ of $\varphi$ such that $\mu=\varphi^n+\eta^{\ast}\eta : X \rightarrow X$ is invertible.
In this case,
$$\varphi^{\co}=\varphi^{n-1}\mu^{-1}.$$
\end{thm}

\begin{proof}
Suppose that $\varphi$ is core invertible with core inverse $\varphi^{\co}$.
Since $(1_X-\varphi\varphi^{\co})\varphi=0$,
then $\eta=1_X-\varphi\varphi^{\co}$ is a annihilator of $\varphi$ such that $\eta^{\ast}=\eta=\eta^2$.
In this case,
$\mu=\varphi^n+\eta^{\ast}\eta=\varphi^n+1_X-\varphi\varphi^{\co}$.
Since
$$(\varphi^n+1_X-\varphi\varphi^{\co})((\varphi^{\co})^n+1_X-\varphi^{\co}\varphi)=1_X=((\varphi^{\co})^n+1_X-\varphi^{\co}\varphi)(\varphi^n+1_X-\varphi\varphi^{\co}),$$
$\mu=\varphi^n+\eta^{\ast}\eta$ is invertible.

Conversely,
if there exists an annihilator $\eta : N \rightarrow X$ of $\varphi$ such that $\mu=\varphi^n+\eta^{\ast}\eta : X \rightarrow X$ is invertible,
where $n\geqslant 2$ is a positive integer.
On the one hand,
$\varphi^{\ast}\mu=\varphi^{\ast}\varphi^n$,
which implies $\varphi^{\ast}=\varphi^{\ast}\varphi^n\mu^{-1}$,
so $\varphi=(\mu^{-1})^{\ast}(\varphi^{\ast})^n\varphi$.
On the other hand,
$\mu\varphi=\varphi^{n+1}$ shows that $\varphi=\mu^{-1}\varphi^{n+1}$.
Therefore,
$\varphi$ is core invertible with $\varphi^{\co}=\varphi^{n-1}\mu^{-1}$ by Lemma~\ref{core-inverse 2}.
\end{proof}

\begin{cor}
Let $a\in R$ and $n\geqslant 2$ a positive integer,
then $a\in R^{\co}$ if and only if there exists $b\in$ $^{\circ}\!a$ such that $u=a^n+b^{\ast}b$ is invertible.
In this case,
$$a^{\co}=a^{n-1}u^{-1}.$$
\end{cor}

Analogously, there are similar conclusions for dual core inverses, which are not to be repeated here.

\section{Core and Dual Core Inverses of  Regular Morphisms with Kernels and Cokernels}\label{w}
In \cite{MR},
Miao and Robinson investigated the group and Moore-Penrose inverses of regular morphisms with kernels and cokernels,
and they showed us two results as follows.
Let $\varphi : X \rightarrow Y$ be a morphism with kernel $\kappa : K \rightarrow X$ and cokernel $\lambda : Y \rightarrow L$ in an additive category $\mathscr{C}$.
(1) If $X=Y$,
then $\varphi$ has a group inverse $\varphi^{\#}$ if and only if $\varphi$ is regular and $\kappa\lambda$ is invertible.
(2) $\varphi$ has a Moore-Penrose inverse $\varphi^{\dagger}$ if and only if $\varphi$ is regular and both $\kappa\kappa^{\ast}$ and $\lambda^{\ast}\lambda$ are invertible.

Inspired by them,
we investigate the core and dual core inverses of regular morphisms with kernels and cokernels.

\begin{lem}\label{131313}
Let $\varphi : X \rightarrow Y$ be a morphism with kernel $\kappa : K \rightarrow X$,
then $\varphi$ is $\{1, 3\}$-invertible if and only if $\varphi$ is regular and $\kappa\kappa^{\ast} : K \rightarrow K$ is invertible.
In this case,
if $\psi : Y \rightarrow X$ is such that $\varphi\psi\varphi=\varphi$,
then
\begin{equation*}
\begin{split}
     \psi[1_X-\kappa^{\ast}(\kappa\kappa^{\ast})^{-1}\kappa]\in \varphi\{1, 3\}.
\end{split}
\end{equation*}
\end{lem}

\begin{proof}
Suppose that $\varphi$ is $\{1, 3\}$-invertible with $\{1, 3\}$-inverse $\varphi^{(1, 3)}$.
Since $\varphi\varphi^{(1, 3)}\varphi=\varphi$,
then $\varphi$ is regular.
Moreover,
since $(1_X-\varphi\varphi^{(1, 3)})\varphi=0$,
then by the definition of a kernel,
$1_X-\varphi\varphi^{(1, 3)}=\zeta\kappa$ for some $\zeta : X \rightarrow K$.
In addition,
since $\kappa\varphi=0$,
then $\kappa\zeta\kappa=\kappa(1_X-\varphi\varphi^{(1, 3)})=\kappa$,
and since $\kappa$ is monic,
$\kappa\zeta=1_K$.
Therefore,
$\zeta\kappa\zeta=\zeta$.
Consequently,
$\kappa$ is Moore-Penrose invertible with $\kappa^{\dagger}=\zeta$.
Since,
\begin{equation*}
\begin{split}
     (\kappa\kappa^{\ast})(\zeta^{\ast}\zeta)=\kappa(\zeta\kappa)^{\ast}\zeta=\kappa\zeta\kappa\zeta=1_K
\end{split}
\end{equation*}
and $\kappa\kappa^{\ast}$ is symmetric,
then $\kappa\kappa^{\ast}$ is invertible with $(\kappa\kappa^{\ast})^{-1}=\zeta^{\ast}\zeta=(\kappa^{\ast})^{\dagger}\kappa^{\dagger}$.

Conversely,
suppose that $\varphi\psi\varphi=\varphi$ and that $\kappa\kappa^{\ast} : K \rightarrow K$ is invertible,
where $\psi : Y \rightarrow X$ is morphism.
We prove that $\chi=\psi[1_X-\kappa^{\ast}(\kappa\kappa^{\ast})^{-1}\kappa]$ is a $\{1, 3\}$-inverse of $\varphi$.
Indeed,
since $(1_X-\varphi\psi)\varphi=0$,
then $1_X-\varphi\psi=\delta\kappa$ for some $\delta : X \rightarrow K$.
Therefore,
\begin{equation*}
\begin{split}
     \varphi\chi
     &~= \varphi\psi[1_X-\kappa^{\ast}(\kappa\kappa^{\ast})^{-1}\kappa] = (1_X-\delta\kappa)[1_X-\kappa^{\ast}(\kappa\kappa^{\ast})^{-1}\kappa]\\
     &~= 1_X-\delta\kappa-\kappa^{\ast}(\kappa\kappa^{\ast})^{-1}\kappa+\delta\kappa\kappa^{\ast}(\kappa\kappa^{\ast})^{-1}\kappa\\
     &~= 1_X-\kappa^{\ast}(\kappa\kappa^{\ast})^{-1}\kappa
\end{split}
\end{equation*}
is symmetric.
Furthermore,
Since $\kappa\varphi=0$,
then
\begin{equation*}
\begin{split}
     \varphi\chi\varphi=(1_X-\kappa^{\ast}(\kappa\kappa^{\ast})^{-1}\kappa)\varphi=\varphi.
\end{split}
\end{equation*}
Thus,
$\psi[1_X-\kappa^{\ast}(\kappa\kappa^{\ast})^{-1}\kappa]\in \varphi\{1, 3\}.$
\end{proof}

Similarly,
we have the following result.

\begin{lem}
Let $\varphi : X \rightarrow Y$ be a morphism with cokernel $\lambda : Y \rightarrow L$,
then $\varphi$ is $\{1, 4\}$-invertible if and only if $\varphi$ is regular and $\lambda^{\ast}\lambda : L \rightarrow L$ is invertible.
In this case,
if $\psi : Y \rightarrow X$ is such that $\varphi\psi\varphi=\varphi$,
then
\begin{equation*}
\begin{split}
     [1_X-\lambda(\lambda^{\ast}\lambda)^{-1}\lambda^{\ast}]\psi\in \varphi\{1, 4\}.
\end{split}
\end{equation*}
\end{lem}

\begin{thm}\label{kernel222}
Let $\varphi : X \rightarrow X$ be a morphism with kernel $\kappa : K \rightarrow X$ and cokernel $\lambda : X \rightarrow L$ in an additive category $\mathscr{C}$,
then $\varphi$ has a core inverse in $\mathscr{C}$ if and only if $\varphi$ is regular and both $\kappa\lambda : K \rightarrow L$ and $\kappa\kappa^{\ast} : K \rightarrow K$ are invertible.
In this case,
if $\psi : X \rightarrow X$ is such that $\varphi\psi\varphi=\varphi$,
then
\begin{equation*}
\begin{split}
     \varphi^{\co}=[1_X-\lambda(\kappa\lambda)^{-1}\kappa]\psi[1_X-\kappa^{\ast}(\kappa\kappa^{\ast})^{-1}\kappa].
\end{split}
\end{equation*}
\end{thm}

\begin{proof}
By Lemma~\ref{222},
Lemma~\ref{131313} and \cite[Theorem]{MR},
it is clear that $\varphi$ is core invertible if and only if $\varphi$ is regular and both $\kappa\lambda : K \rightarrow L$ and $\kappa\kappa^{\ast} : K \rightarrow K$ are invertible.

Suppose that $\psi : X \rightarrow X$ is such that $\varphi\psi\varphi=\varphi$,
we show that $\chi=[1_X-\lambda(\kappa\lambda)^{-1}\kappa]\psi[1_X-\kappa^{\ast}(\kappa\kappa^{\ast})^{-1}\kappa]$ is the core inverse of $\varphi$.
Indeed,
since $(1_X-\varphi\psi)\varphi=0=\varphi(1_X-\psi\varphi)$,
then $1_X-\varphi\psi=\delta\kappa$ and $1_X-\psi\varphi=\lambda\zeta$ for some $\delta : X \rightarrow K$ and $\zeta : L \rightarrow X$,
respectively.
Since $\kappa\varphi=0=\varphi\lambda$,
then
\begin{equation*}
\begin{split}
     \varphi\chi
     &~= \varphi[1_X-\lambda(\kappa\lambda)^{-1}\kappa]\psi[1_X-\kappa^{\ast}(\kappa\kappa^{\ast})^{-1}\kappa]\\
     &~= \varphi\psi[1_X-\kappa^{\ast}(\kappa\kappa^{\ast})^{-1}\kappa]\\
     &~= (1_X-\delta\kappa)[1_X-\kappa^{\ast}(\kappa\kappa^{\ast})^{-1}\kappa]\\
     &~= 1_X-\kappa^{\ast}(\kappa\kappa^{\ast})^{-1}\kappa
\end{split}
\end{equation*}
is symmetric.
In addition,
\begin{equation*}
\begin{split}
     \chi\varphi^2
     &~= [1_X-\lambda(\kappa\lambda)^{-1}\kappa]\psi[1_X-\kappa^{\ast}(\kappa\kappa^{\ast})^{-1}\kappa]\varphi^2\\
     &~= [1_X-\lambda(\kappa\lambda)^{-1}\kappa]\psi\varphi^2 = [1_X-\lambda(\kappa\lambda)^{-1}\kappa](1_X-\lambda\zeta)\varphi\\
     &~= [1_X-\lambda(\kappa\lambda)^{-1}\kappa]\varphi = \varphi
\end{split}
\end{equation*}
and
\begin{equation*}
\begin{split}
     \varphi\chi^2
     &~= [1_X-\kappa^{\ast}(\kappa\kappa^{\ast})^{-1}\kappa][1_X-\lambda(\kappa\lambda)^{-1}\kappa]\psi[1_X-\kappa^{\ast}(\kappa\kappa^{\ast})^{-1}\kappa]\\
     &~= [1_X-\lambda(\kappa\lambda)^{-1}\kappa]\psi[1_X-\kappa^{\ast}(\kappa\kappa^{\ast})^{-1}\kappa = \chi.
\end{split}
\end{equation*}
Therefore,
$\varphi$ is core invertible with core inverse $\varphi^{\co}=[1_X-\lambda(\kappa\lambda)^{-1}\kappa]\psi[1_X-\kappa^{\ast}(\kappa\kappa^{\ast})^{-1}\kappa]$.
\end{proof}

Similarly,
we can get a dually result about dual core inverse.
\begin{thm}\label{kernel222}
Let $\varphi : X \rightarrow X$ be a morphism with kernel $\kappa : K \rightarrow X$ and cokernel $\lambda : X \rightarrow L$ in an additive category $\mathscr{C}$,
then $\varphi$ has a dual core inverse in $\mathscr{C}$ if and only if $\varphi$ is regular and both $\kappa\lambda : K \rightarrow L$ and $\lambda^{\ast}\lambda : L \rightarrow L$ are invertible.
In this case,
if $\psi : X \rightarrow X$ is such that $\varphi\psi\varphi=\varphi$,
then
\begin{equation*}
\begin{split}
     \varphi_{\co}=[1_X-\lambda(\lambda^{\ast}\lambda)^{-1}\lambda^{\ast}]\psi[1_X-\lambda(\kappa\lambda)^{-1}\kappa].
\end{split}
\end{equation*}
\end{thm}

\begin{cor}
Let $\varphi : X \rightarrow X$ be a morphism with kernel $\kappa : K \rightarrow X$ and cokernel $\lambda : X \rightarrow L$ in an additive category $\mathscr{C}$,
then the following statements are equivalent:\\
(1) $\varphi$ is both core invertible and dual core invertible in $\mathscr{C}$; \\
(2) $\varphi$ is both Moore-Penrose invertible and group invertible in $\mathscr{C}$; \\
(3) $\varphi$ is regular and $\kappa\lambda : K \rightarrow L$,
$\kappa\kappa^{\ast} : K \rightarrow K$ and $\lambda^{\ast}\lambda : L \rightarrow L$ are all invertible.\\
In this case,
if $\psi : X \rightarrow X$ is such that $\varphi\psi\varphi=\varphi$,
then
\begin{equation*}
\begin{split}
     \varphi^{\dagger}&~=[1_X-\lambda(\lambda^{\ast}\lambda)^{-1}\lambda^{\ast}]\psi[1_X-\kappa^{\ast}(\kappa\kappa^{\ast})^{-1}\kappa],\\
     \varphi^{\#}&~=[1_X-\lambda(\kappa\lambda)^{-1}\kappa]\psi[1_X-\lambda(\kappa\lambda)^{-1}\kappa],\\
     \varphi^{\co}&~=[1_X-\lambda(\kappa\lambda)^{-1}\kappa]\psi[1_X-\kappa^{\ast}(\kappa\kappa^{\ast})^{-1}\kappa],\\
     \varphi_{\co}&~=[1_X-\lambda(\lambda^{\ast}\lambda)^{-1}\lambda^{\ast}]\psi[1_X-\lambda(\kappa\lambda)^{-1}\kappa].
\end{split}
\end{equation*}
\end{cor}

\section{Bordered Inverses}\label{c}

Recall that a morphism $\varphi : X \rightarrow Y$ is $\ast$-left invertible if there is a morphism $\psi : Y \rightarrow X$ such that $\psi\varphi=1_Y$ and $(\varphi\psi)^{\ast}=\varphi\psi$.
Similarly,
$\varphi : X \rightarrow Y$ is $\ast$-right invertible if there is a morphism $\psi : Y \rightarrow X$ such that $\varphi\psi=1_X$ and $(\psi\varphi)^{\ast}=\psi\varphi$.
(See, \cite[p. 76]{DW}.)

\begin{lem}\cite[Lemma]{DW}\label{*-inv.}
If $\varphi : X \rightarrow Y$ is a morphism in a category with an involution.
Then \\
(1) $\varphi$ is $\ast$-left invertible if and only if $\varphi^{\ast}\varphi$ is invertible,
and in this case,
$\varphi^{\dagger}=(\varphi^{\ast}\varphi)^{-1}\varphi^{\ast}$;\\
(2) $\varphi$ is $\ast$-right invertible if and only if $\varphi\varphi^{\ast}$ is invertible,
and in this case,
$\varphi^{\dagger}=\varphi^{\ast}(\varphi\varphi^{\ast})^{-1}$.
\end{lem}

\begin{lem}\cite[Corollary $3$]{DW}\label{*-inv. group}
Let $\varphi : X \rightarrow X$ be a morphism of an additive category $\mathscr{C}$.
If $\kappa : K \rightarrow X$ is a kernel and $\lambda : X \rightarrow L$ is a cokernel of $\varphi$,
then $\varphi$ has a group inverse in $\mathscr{C}$ if and only if
$$\mathscr{G}=\left(
                            \begin{array}{cc}
                              \varphi & \lambda \\
                              \kappa & 0  \\
                            \end{array}
                          \right) : (X, K) \rightarrow (X, L)$$
is invertible in $\mathscr{M_C}$.
In this case,
$\kappa\lambda : K \rightarrow L$ is invertible and
$$\mathscr{G}^{-1}=\left(
                            \begin{array}{cc}
                              \varphi^{\#} & \lambda(\kappa\lambda)^{-1} \\
                              (\kappa\lambda)^{-1}\kappa & 0  \\
                            \end{array}
                          \right) : (X, L) \rightarrow (X, K).$$
\end{lem}

\begin{thm}
Let $\varphi : X \rightarrow X$ be a morphism of an additive category $\mathscr{C}$ with an involution $\ast$.
If $\kappa : K \rightarrow X$ is a kernel of $\varphi$ and $\lambda : X \rightarrow L$ is a cokernel of $\varphi$,
then $\varphi$ has a core inverse $\varphi^{\co}$ in $\mathscr{C}$ if and only if
$$\left(
                            \begin{array}{c}
                              \varphi \\
                              \kappa \\
                            \end{array}
                          \right) : (X, K) \rightarrow (X)$$
is $\ast$-left invertible in $\mathscr{M_C}$ and
$$\mathscr{G}=\left(
                            \begin{array}{cc}
                              \varphi & \lambda \\
                              \kappa & 0  \\
                            \end{array}
                          \right) : (X, K) \rightarrow (X, L)$$
is invertible in $\mathscr{M_C}$.
In this case,
$\kappa\lambda : K \rightarrow L$ is invertible and
$$\mathscr{G}^{-1}=\left(
                            \begin{array}{cc}
                              (\varphi^{\co})^2\varphi & \lambda(\kappa\lambda)^{-1} \\
                              (\kappa\lambda)^{-1}\kappa & 0  \\
                            \end{array}
                          \right) : (X, L) \rightarrow (X, K).$$
\end{thm}

\begin{proof}
By Theorem~\ref{kernel},
$\varphi$ has a core inverse in $\mathscr{C}$ if and only if both  $\kappa\lambda$ and $\varphi^{\ast}\varphi^3+\kappa^{\ast}\kappa$ are invertible.
Since $\varphi^{\ast}\varphi^3+\kappa^{\ast}\kappa=(\varphi^{\ast}\varphi+\kappa^{\ast}\kappa)(\varphi^2+\lambda(\kappa\lambda)^{-1}\kappa)$ and $\varphi^{\ast}\varphi+\kappa^{\ast}\kappa$ is symmetric,
then $\varphi^{\ast}\varphi^3+\kappa^{\ast}\kappa$ is invertible if and only if $\varphi^{\ast}\varphi+\kappa^{\ast}\kappa$ and $\varphi^2+\lambda(\kappa\lambda)^{-1}\kappa$ are both invertible.
By Lemma~\ref{*-inv.},
$\varphi^{\ast}\varphi+\kappa^{\ast}\kappa$ is invertible if and only if $\left(
                            \begin{array}{c}
                              \varphi \\
                              \kappa \\
                            \end{array}
                          \right)^{\ast}
\left(
                            \begin{array}{c}
                              \varphi \\
                              \kappa \\
                            \end{array}
                          \right)$ is invertible in $\mathscr{C}$ if and only if
$\left(
                            \begin{array}{c}
                              \varphi \\
                              \kappa \\
                            \end{array}
                          \right)$ is $\ast$-left invertible in $\mathscr{M_C}$.
Therefore,
the conclusion is obtained by the previous proof,
Lemma~\ref{kernel-group inverse} and Lemma~\ref{*-inv. group}.
And it is easy to verify that $\left(
                            \begin{array}{cc}
                              (\varphi^{\co})^2\varphi & \lambda(\kappa\lambda)^{-1} \\
                              (\kappa\lambda)^{-1}\kappa & 0  \\
                            \end{array}
                          \right) : (X, L) \rightarrow (X, K)$ is the inverse of $\mathscr{G}$.
\end{proof}

We have a dually theorem about dual core inverse.
\begin{thm}
Let $\varphi : X \rightarrow X$ be a morphism of an additive category $\mathscr{C}$ with an involution $\ast$.
If $\kappa : K \rightarrow X$ is a kernel of $\varphi$ and $\lambda : X \rightarrow L$ is a cokernel of $\varphi$,
then $\varphi$ has a dual core inverse $\varphi_{\co}$ in $\mathscr{C}$ if and only if
$$(\varphi, \lambda) : (X) \rightarrow (X, L)$$
is $\ast$-right invertible and
$$\mathscr{G}=\left(
                            \begin{array}{cc}
                              \varphi & \lambda \\
                              \kappa & 0  \\
                            \end{array}
                          \right) : (X, K) \rightarrow (X, L)$$
is invertible in $\mathscr{M_C}$.
In this case,
$\kappa\lambda : K \rightarrow L$ is invertible and
$$\mathscr{G}^{-1}=\left(
                            \begin{array}{cc}
                              \varphi\varphi_{\co}^2 & \lambda(\kappa\lambda)^{-1} \\
                              (\kappa\lambda)^{-1}\kappa & 0  \\
                            \end{array}
                          \right) : (X, L) \rightarrow (X, K).$$

\end{thm}

\vspace{0.2cm} \noindent {\large\bf Acknowledgements}

This research is supported by the National Natural Science Foundation of China (No.11771076 and No.11471186); 
the Fundamental Research Funds For the Central Universities (No.KYCX17\_0037);
Postgraduate Research \& Practice Innovation Program of Jiangsu Province (No.KYCX17\_0037).

\end{document}